\documentclass[12pt,a4paper]{amsart}

\usepackage{amsmath,amssymb}
\usepackage{enumerate}
\usepackage[margin=2cm]{geometry}
\usepackage{hyperref}

\newtheorem{theorem}{Theorem}
\numberwithin{theorem}{section}
\newtheorem{lemma}[theorem]{Lemma}
\newtheorem{proposition}[theorem]{Proposition}
\newtheorem{corollary}[theorem]{Corollary}

\theoremstyle{definition}
\newtheorem{definition}[theorem]{Definition}
\newtheorem{example}[theorem]{Example}

\title{On the rational solutions of generalized Abel equations}

\author{L.A. Calder\'on*}
\address{Departamento de Matemáticas, Universidad de Castilla-La Mancha, 13071 Ciudad Real, Spain}
\email{luisangel.calderon@uclm.es}
\author{I. Ojeda}
\address{Departamento de Matemáticas, Universidad de Extremadura, 06071 Badajoz, Spain}
\email{ojedamc@unex.es}

\thanks{* Corresponding author}

\begin{document}
	
	\begin{abstract}
		We study the nonconstant rational solutions of generalized Abel equations
		\[
		x'=\sum_{i=1}^s A_i(t)x^{n_i},\qquad s\geq3,\qquad 1\leq n_1<\cdots<n_s,
		\]
		with polynomial coefficients over $\mathbb{R}$ or $\mathbb{C}$. Using the Newton diagram at infinity and the associated edge polynomials, we obtain, under simplicity and nonresonance conditions on the edge polynomials, the global upper bound $n_s-1$ for the number of nonconstant rational solutions. Over $\mathbb{R}$, the same approach yields the additional bound $\min\{n_s-1,2s\}$. We construct a family of equations showing that both bounds are sharp for every $s\geq3$. When $n_1=1$, the real bound improves to $\min\{n_s-1,2(s-1)\}$. We then specialize to the three-term case with $1<n_1<n_2<n_3$, where we obtain a finer classification of the rational solutions. Under an additional antipodal separation condition, we obtain the real bound $\min\{n_3-1,5\}$, and show that the absolute constant $5$ is sharp.
	\end{abstract}
	
	
	
	\thanks{
		Both authors are partially supported by the project PID2023-151974NB-I00 funded by MICIU/
		AEI/10.13039/501100011033/FEDER, UE. This work has been partially funded by the Junta de Extremadura through projects GR24042 (L-A.C.) and GR24068 (I.O.), partially funded by the European Regional Development Fund (ERDF) ``A way to make Europe''.}
	
	\maketitle
	
	\section{Introduction}
	
	Abel differential equations and their generalizations have been studied from several perspectives in the qualitative theory of differential equations. In this paper, we focus on the existence and number of rational solutions for equations with polynomial coefficients.
	
	Polynomial solutions constitute a first natural class of explicitly describable solutions. When they define isolated periodic solutions, they give rise to polynomial limit cycles. Giné, Grau, and Llibre~\cite{GineGrauLlibre2011} proved that an equation $y'=P(x,y)$, with $P\in\mathbb{R}[x,y]$ of degree $n$ in $y$, has at most $n$ polynomial limit cycles, and that this bound is sharp. Polynomial solutions of Abel and related equations have also been studied in~\cite{CimaGasullManosas2017,LlibreValls2018}, while the Riccati case is considered in~\cite{GasullTorregrosaZhang2016}.
	
	Rational solutions form a broader class and admit a geometric interpretation in terms of invariant algebraic curves. A rational solution $x(t)=p(t)/q(t)$ determines the invariant algebraic curve $q(t)x-p(t)=0$ of the planar polynomial system associated with the differential equation. For the two-term Abel equation
	\[
	x'=A(t)x^3+B(t)x^2,
	\]
	rational solutions and their number have been studied in~\cite{BravoCalderonFernandezOjeda2022,QianShenYang2021}. Rational limit cycles of this equation have also been considered in the polynomial setting in~\cite{LiuLiWangWu2018,LlibreValls2021} and in the trigonometric setting in~\cite{BravoCalderonOjeda2023,Valls2022}. Related upper bounds for periodic Abel equations with invariant curves are obtained in~\cite{BravoCalderonFernandez2021}. The classical Abel equation with a linear term has been studied in~\cite{Calderon2025}, in both the polynomial and trigonometric settings. Du and Zhao~\cite{DuZhao2026} obtain bounds for equations with two arbitrary nonlinear terms, while Zeng and Zhao~\cite{ZengZhaoUpperBound} treat two nonlinear terms together with a linear term.
	
	We consider generalized Abel equations of the form
	\begin{equation}\label{ecu:main}
		x'=A_s(t)x^{n_s}+\dots+A_1(t)x^{n_1},
		\qquad 1\leq n_1<\cdots<n_s,
	\end{equation}
	where $s\geq3$, $A_i\in\Bbbk[t]$ is nonzero for every $i$, and $\Bbbk\in\{\mathbb{R},\mathbb{C}\}$. Our principal new setting is that in which all terms are nonlinear, namely $n_1>1$.
	
	Using the Newton diagram at infinity and formal Laurent expansions, within the differential Newton--Puiseux framework of~\cite{Cano2005NewtonPolygon,DeminaPuiseux2022}, we identify the possible dominant balances of the equation at infinity and use them to constrain the leading terms of rational solutions. Under explicit simplicity and nonresonance conditions, see Definition~\ref{def:SN}, this approach yields bounds independent of the degrees of the polynomial coefficients. These conditions are expressed explicitly in terms of the exponents $n_i$, the degrees of the $A_i$, and their leading coefficients. For each $s\geq3$, we determine the exact maximal number of nonconstant real rational solutions, as the exponents and coefficients vary within the class satisfying these hypotheses. We also derive restrictions on the denominator degrees and establish a composition structure when a proportionality class attains its maximal size. The same arguments apply when $n_1=1$ and yield a refinement of the real bound in that case.
	
	By Proposition~\ref{prop:condinv}, every nonconstant rational solution of equation~\eqref{ecu:main} is of the form $x=1/p$, where $p\in\Bbbk[t]$ and $p^{n_s-n_{s-1}}$ divides $A_s$. This divisibility condition restricts the possible denominators, up to nonzero constant factors, to finitely many divisors of $A_s$. It therefore provides useful information for individual equations, although the number of candidates depends on the factorization of $A_s$ and can grow with its degree. To obtain bounds independent of the degrees of the coefficient polynomials, we complement this algebraic restriction with an analysis of the possible leading terms of rational solutions at infinity.
	
	To this end, write a rational solution as $x=1/p$, set $r=\deg p$, and observe that its Laurent expansion at infinity has the form
	\[
	x(t)=Ct^{-r}+O(t^{-r-1}),\qquad C\neq0.
	\]
	Only those values of $r$ for which at least two dominant terms of the differential equation balance at infinity can occur. Such a balance determines an exposed edge of the Newton diagram and an associated edge polynomial $P_r$, and the possible leading coefficients $C$ of rational solutions with denominator degree $r$ are nonzero roots of $P_r$.
	
	Under the simplicity condition {\rm (S)}, imposed only on edges not containing the point that represents the derivative term, and the nonresonance condition {\rm (NR)}, imposed only on edges containing that point, the recursive determination of the Laurent coefficients is nondegenerate. Consequently, each such nonzero root of the edge polynomial determines a unique formal Laurent solution, and distinct rational solutions associated with the same edge must correspond to distinct nonzero roots of that polynomial. For completeness, and to make the counting argument self-contained, we include proofs of both the numerator reduction and the formal Laurent recursion used below.
	
	Using this edgewise injectivity and the geometry of the Newton diagram, we prove in Theorem~\ref{thm:global-bound} that equation~\eqref{ecu:main} has at most $n_s-1$ nonconstant rational solutions under conditions {\rm (S)} and {\rm (NR)}. For each exposed edge, the number of rational solutions is bounded by the number of its possible nonzero leading coefficients, and hence by the difference between the largest and smallest exponents represented on that edge. Summing these differences along the relevant boundary chain gives the global bound $n_s-1$. Over $\mathbb{R}$, Descartes' rule of signs gives the additional bound $\min\{n_s-1,2s\}$. We further show that if $n_1>1$ and a proportionality class attains the maximal size $n_s-1$, then the coefficients necessarily satisfy a composition condition; see Proposition~\ref{prop:composition}. For every $s\geq3$, we construct an equation satisfying these hypotheses with exactly $2s$ nonconstant real rational solutions. Thus $2s$ is the optimal bound depending only on the number of terms in this class. The same examples attain the complex bound $n_s-1$, with $n_s=2s+1$. When $n_1=1$, Corollary~\ref{cor:linear-term-bound} gives the real bound $\min\{n_s-1,2(s-1)\}$ under the simplicity condition alone.
	
	The simplicity {\rm (S)} and nonresonance {\rm (NR)} assumptions ensure that each admissible leading term determines a unique formal Laurent solution, which is the basis of our counting argument. The cases excluded by these hypotheses can be investigated further using Newton--Puiseux techniques, but this requires examining additional compatibility conditions and possible branching to determine which formal solutions are rational. Such phenomena already arise in the two-term setting studied in~\cite{DuZhao2026}. We do not pursue this analysis here, focusing instead on uniform bounds and their sharpness under explicit hypotheses ensuring uniqueness.
	
	Lastly, we specialize to equations with three nonlinear terms,
	\[
	x'=A_3(t)x^{n_3}+A_2(t)x^{n_2}+A_1(t)x^{n_1},
	\qquad 1<n_1<n_2<n_3.
	\]
	In this setting, the Newton diagram also constrains the denominator degrees and the distribution of rational solutions among them, and the reduced number of exponents allows us to determine an exhaustive classification of the realized denominator degrees.
	
	Without simplicity or nonresonance assumptions, the denominator polynomials $p$ occurring in rational solutions $x=1/p$ can have at most three distinct degrees; see Corollary~\ref{cor:at-most-three-degrees}. Under {\rm (NR)}, the denominators of all rational solutions can realize at most two distinct degrees, and we identify the possible edge configurations; see Corollary~\ref{cor:at-most-two-degrees} and Proposition~\ref{prop:two-degree-configurations}. We also show that, if the edge associated with the maximal realized denominator degree contains the three nonlinear exponent points but not the derivative point, then all nonconstant rational solutions have that denominator degree and their total number is at most $n_3-n_1$; see Corollary~\ref{cor:three-nonlinear-edge}.
	
	Over $\mathbb{R}$, these restrictions refine the global bound of six. Under {\rm (S)} and {\rm (NR)}, the denominator-degree multiplicity profile is componentwise bounded by one of
	\[
	(6),\qquad (2,2),\qquad (2,4).
	\]
	Here the entries refer to the realized denominator degrees listed in increasing order. For example, the profile bound $(2,4)$ means that, if two degrees $d_1<d_2$ occur, then at most two rational solutions have denominator degree $d_1$ and at most four have denominator degree $d_2$; see Theorem~\ref{thm:real-basic-profiles}. Thus the analysis describes both the total number of solutions and their possible distribution among denominator degrees.
	
	Finally, we introduce an antipodal separation condition {\rm (AS)}, see Definition~\ref{def:antipodal-separation}, which excludes the relevant configurations in which nonzero real roots of an edge polynomial occur in opposite pairs. Under this additional assumption, the profile bounds improve to
	\[
	(5),\qquad (1,2),\qquad (1,4),
	\]
	and the total number of nonconstant real rational solutions is at most $\min\{n_3-1,5\}$. Example~\ref{ex:sharp-real} attains five solutions, proving that the absolute constant remains sharp after imposing {\rm (AS)}. Moreover, if five solutions occur, they either all have the same denominator degree or are distributed between two degrees with multiplicity profile $(1,4)$; see Corollary~\ref{cor:five-solutions}.
	
	\section{Preliminaries}\label{Sect:2}
	
	Throughout Sections~\ref{Sect:2} and~\ref{Sect:3}, we work over a base field $\Bbbk\in\{\mathbb{R},\mathbb{C}\}$ and consider equation~\eqref{ecu:main}, 
	\[ 
	x'=A_s(t)x^{n_s}+\dots+A_1(t)x^{n_1},\qquad 1\leq n_1<\cdots<n_s, 
	\] 
	where $s\geq3$ and $A_i(t)\in\Bbbk[t]$ is nonzero for every $i\in\{1,\ldots,s\}$. The classical two-term Abel equation was studied in~\cite{BravoCalderonFernandezOjeda2022}, while the case of two arbitrary nonlinear exponents was considered in~\cite{DuZhao2026}; this is why we restrict our attention to $s\geq3$. We write $a_i:=\deg A_i$ and $\alpha_i:=\operatorname{lc}(A_i)\in\Bbbk^\times$ for every $i\in\{1,\ldots,s\}$, where $\operatorname{lc}(\cdot)$ denotes the leading coefficient with respect to $t$.
	
	We begin by characterizing the nonzero rational solutions of equation~\eqref{ecu:main}. The corresponding numerator reduction for classical Abel equations appears in~\cite{BravoCalderonFernandezOjeda2022}; below we give its extension to the general finite-term setting.
	
	\begin{proposition}\label{prop:condinv}
		Let $p,q\in\Bbbk[t]$ be nonzero coprime polynomials, and assume that $q$ is monic. If $x=q/p$ is a solution of equation~\eqref{ecu:main}, then $q\equiv1$ and
		\begin{equation}\label{eq:denominator}
			p^{n_s-2}p'+\sum_{i=1}^s A_i p^{n_s-n_i}=0.
		\end{equation}
		Conversely, every nonzero polynomial $p\in\Bbbk[t]$ satisfying~\eqref{eq:denominator} defines a nonzero rational solution $x=1/p$ of equation~\eqref{ecu:main}.
	\end{proposition}
	
	\begin{proof}
		Suppose that $x=q/p$ solves equation~\eqref{ecu:main}. Omitting the argument $t$, substitution and multiplication by $p^{n_s}$ give
		\begin{equation}\label{eq:prop1}
			p^{n_s-2}(q'p-qp')=\sum_{i=1}^s A_iq^{n_i}p^{n_s-n_i}.
		\end{equation}
		Since $n_i\geq1$ for every $i\in\{1,\ldots,s\}$, the right-hand side is divisible by $q$, and hence $q\mid p^{n_s-2}(q'p-qp')$. Since $\gcd(p,q)=1$, we also have $\gcd(p^{n_s-2},q)=1$, and therefore $q\mid q'p-qp'$. As $q\mid qp'$, it follows that $q\mid q'p$. Using again $\gcd(p,q)=1$, we obtain $q\mid q'$. This is impossible if $q$ is nonconstant, since $\deg q'<\deg q$. Thus $q$ is constant, and the normalization that $q$ is monic gives $q\equiv1$. Substituting $q\equiv1$ into~\eqref{eq:prop1} yields~\eqref{eq:denominator}.
		
		Conversely, if $p$ satisfies~\eqref{eq:denominator}, then $x=1/p$ satisfies $x'=-p'/p^2=\sum_{i=1}^s A_i/p^{n_i}$, and hence solves equation~\eqref{ecu:main}.
	\end{proof}
	
	In view of Proposition~\ref{prop:condinv}, every nonconstant rational solution of equation~\eqref{ecu:main} can and will be written as $x(t)=1/p(t)$, where $p\in\Bbbk[t]$ is not necessarily monic. Henceforth, by a rational solution we mean a solution of this form with $\deg p\geq1$, thereby excluding constant solutions.
	
	By Proposition~\ref{prop:condinv}, the rational function $x=1/p(t)$ is a solution of equation~\eqref{ecu:main} if and only if
	\begin{equation}\label{ecu:condinv}
		A_s(t)=-p(t)^{n_s-n_{s-1}}\left(p(t)^{n_{s-1}-2}p'(t)+\sum_{i=1}^{s-1}A_i(t)p(t)^{n_{s-1}-n_i}\right).
	\end{equation}
	In particular, $p(t)^{n_s-n_{s-1}}$ divides $A_s(t)$, and therefore $\deg p\leq a_s/(n_s-n_{s-1})$. As an immediate consequence, we obtain the following finiteness result.
	
	\begin{corollary}\label{cor:finiteness}
		Equation~\eqref{ecu:main} has finitely many nonconstant rational solutions.
	\end{corollary}
	
	\begin{proof}
		By~\eqref{ecu:condinv}, the denominator $p$ of every rational solution satisfies $p^{n_s-n_{s-1}}\mid A_s$. Hence the factorization of $A_s$ yields only finitely many possibilities for $p$ up to multiplication by an element of $\Bbbk^\times$. It therefore remains to prove that each such proportionality class contains only finitely many denominators defining rational solutions.
		
		Let $x=1/p$ be a nonconstant rational solution of equation~\eqref{ecu:main}, and let $\lambda\in\Bbbk^\times$. Then $\lambda x=1/(p/\lambda)$ is also a rational solution if and only if
		\begin{equation}\label{eq:scaled-condition}
			\sum_{i=1}^s\bigl(\lambda^{n_i-1}-1\bigr)A_i p^{n_s-n_i}=0.
		\end{equation}
		The left-hand side of~\eqref{eq:scaled-condition}, regarded as a polynomial in $\lambda$ over $\Bbbk(t)$, has degree $n_s-1$ and leading coefficient $A_s\neq0$. Hence it has at most $n_s-1$ roots in $\Bbbk^\times$. Thus each proportionality class contains at most $n_s-1$ rational solutions, and the result follows.
	\end{proof}
	
	The condition $p^{n_s-n_{s-1}}\mid A_s$ imposes strong restrictions on the existence of rational solutions. In particular, the denominator $p$ must divide $A_s$, and, if $n_s-n_{s-1}>1$, the existence of a nonconstant rational solution implies that $A_s$ has a multiple root over $\mathbb{C}$.
	
	\section{A Newton--Puiseux method for rational solutions}\label{Sect:3}
	
	In this section, we study rational solutions of the form $x(t)=1/p(t)$, with $p(t)\in\Bbbk[t]$, by analyzing equation~\eqref{ecu:main} at infinity through its Newton diagram; see~\cite{Cano2005NewtonPolygon,Demina2021InvariantCurves} for the differential setting and~\cite{Walker1950} for the classical algebraic-geometric background.
	
	All expansions at infinity are understood as Laurent series in $t^{-1}$. In particular, if $x(t)=1/p(t)$ and $r=\deg p$, then $x(t)=Ct^{-r}+O(t^{-r-1})$, where $C=1/\operatorname{lc}(p)\neq0$. By convention, we encode the derivative term $x'(t)$ by the index $\partial$, with associated weights $a_\partial=-1$ and $n_\partial=1$, since $x(t)\sim Ct^{-r}$ implies $x'(t)\sim-rCt^{-r-1}$.
	
	\begin{definition}
		For $r\in\mathbb{N}_{>0}$, set $\Phi_i(r):=n_ir-a_i$ for $i\in\{1,\ldots,s\}$ and $\Phi_\partial(r):=r+1$. Define $O_r:=\min\{\Phi_\ell(r):\ell\in\{1,\ldots,s,\partial\}\}$ and $T_r:=\{\ell\in\{1,\ldots,s,\partial\}:\Phi_\ell(r)=O_r\}$.
	\end{definition}
	
	Notice that, if $n_1=1$, then $\Phi_1(r)=r-a_1<r+1=\Phi_\partial(r)$ for every $r>0$. Hence $\partial\notin T_r$.
	
	We associate with equation~\eqref{ecu:main} the exponent points $Q_i:=(a_i,n_i)$, for $i\in\{1,\ldots,s\}$, and $Q_\partial:=(-1,1)$, and set $\mathcal{Q}:=\{Q_1,\ldots,Q_s,Q_\partial\}$. We define $\mathcal{N}:=\operatorname{Conv}(\mathcal{Q})$ and call the union of the one-dimensional faces exposed by the linear functionals $(a,n)\mapsto-a+nr$, with $r>0$, the Newton diagram at infinity of equation~\eqref{ecu:main}. When $\dim\mathcal N=1$, we regard the whole segment $\mathcal N$ as an exposed edge whenever the functional $(a,n)\mapsto-a+nr$ is constant on $\mathcal N$ for some $r>0$. Terms of lower degree in the polynomials $A_i$ do not affect this diagram, since $n_ir-k>n_ir-a_i$ for every $r>0$, every $i\in\{1,\ldots,s\}$, and every $k<a_i$.
	
	\begin{definition}
		We call $r\in\mathbb{N}_{>0}$ edge-admissible for equation~\eqref{ecu:main} if $|T_r|\geq2$. Equivalently, the Newton diagram has a unique edge $E_r$ with supporting line of slope $1/r$. In this case, $T_r=\{\ell\in\{1,\ldots,s,\partial\}:Q_\ell\in E_r\}$. 
	\end{definition}
	
	If $i,j\in T_r$ are distinct, then $\Phi_i(r)=\Phi_j(r)$ and hence $r=(a_i-a_j)/(n_i-n_j)$.
	
	\begin{definition}
		Let $r\in\mathbb{N}_{>0}$ be edge-admissible. The edge polynomial associated with $E_r$ is
		\[
		P_r(C):=\sum_{\ell\in T_r\cap\{1,\ldots,s\}}\alpha_\ell C^{n_\ell}+\mathbf{1}_{\{\partial\in T_r\}}rC,
		\]
		where $\mathbf{1}_{\{\partial\in T_r\}}$ equals $1$ if $\partial\in T_r$ and $0$ otherwise.
	\end{definition}
	
	\begin{lemma}\label{lemma:lc}
		Let $r\in\mathbb{N}_{>0}$, and let $x(t)=Ct^{-r}+O(t^{-r-1})$, with $C\neq0$, be a formal Laurent series solution of equation~\eqref{ecu:main}. Then $r$ is edge-admissible and $P_r(C)=0$.
	\end{lemma}
	
	\begin{proof}
		The leading terms of the derivative and the terms on the right-hand side of equation~\eqref{ecu:main} are $x'(t)=-rCt^{-r-1}+O(t^{-r-2})$ and $A_i(t)x(t)^{n_i}=\alpha_iC^{n_i}t^{a_i-n_ir}+O\bigl(t^{a_i-n_ir-1}\bigr)$, respectively. Thus the dominant order is determined by $O_r=\min\{\Phi_\ell(r):\ell\in\{1,\ldots,s,\partial\}\}$. If $|T_r|=1$, exactly one nonzero term occurs at order $t^{-O_r}$, which is impossible because its coefficient is either $-rC\neq0$ or $\alpha_iC^{n_i}\neq0$. Hence $|T_r|\geq2$, so $r$ is edge-admissible. Equating the coefficient of $t^{-O_r}$ in equation~\eqref{ecu:main} gives
		\[
		-rC\,\mathbf{1}_{\{\partial\in T_r\}}=\sum_{\ell\in T_r\cap\{1,\ldots,s\}}\alpha_\ell C^{n_\ell},
		\]
		which is equivalent to $P_r(C)=0$.
	\end{proof}
	
	\begin{proposition}\label{prop:leading_coeff_root}
		Let $x(t)=1/p(t)$ be a rational solution of equation~\eqref{ecu:main}, where $p\in\Bbbk[t]$, $\deg p=r$, and $\operatorname{lc}(p)=c$. Then $r$ is edge-admissible and $P_r(1/c)=0$.
	\end{proposition}
	
	\begin{proof}
		Since $p$ has degree $r$ and leading coefficient $c$, its reciprocal has the Laurent expansion $1/p(t)=(1/c)t^{-r}+O(t^{-r-1})$ at infinity. Hence $x(t)$ is a formal Laurent series solution of equation~\eqref{ecu:main} with leading exponent $-r$ and leading coefficient $1/c$. Lemma~\ref{lemma:lc} then implies that $r$ is edge-admissible and $P_r(1/c)=0$.
	\end{proof}
	
	The following proposition specializes the standard formal Newton--Puiseux recursion for first-order algebraic differential equations; see~\cite{Cano2005NewtonPolygon,DeminaPuiseux2022}. For generalized Abel equations with two nonlinear terms, the corresponding recurrence and its use in counting rational solutions appear in~\cite[Appendix~A and Section~3]{DuZhao2026}.
	We include the proof in order to identify explicitly, in the present setting, the recursion coefficient
	\[
	P_r'(C)+m\mathbf{1}_{\{\partial\in T_r\}},
	\]
	which yields the precise simplicity and nonresonance conditions used in the counting argument below. When $\partial\in T_r$, the transformed equation is of Briot--Bouquet type, and the uniqueness statement is also consistent with~\cite[Proposition~2.1]{Rong2013}.
	
	\begin{proposition}\label{prop:existence}
		Let $r\in\mathbb{N}_{>0}$ be edge-admissible, and let $C\in\Bbbk^\times$ be a root of $P_r$. Assume that $P_r'(C)\neq0$ if $\partial\notin T_r$, and that $P_r'(C)\notin\mathbb{Z}_{<0}$ if $\partial\in T_r$. Then there exists a unique formal Laurent series solution of equation~\eqref{ecu:main} of the form
		\[
		x(t)=\sum_{m\geq0}c_mt^{-r-m},\qquad c_0=C.
		\]
	\end{proposition}
	
	\begin{proof}
		Let $F(x):=x'-\sum_{i=1}^sA_i(t)x^{n_i}$, so that equation~\eqref{ecu:main} is equivalent to $F(x)=0$. For each $N\geq0$, set
		\[
		x_N(t):=Ct^{-r}+\sum_{m=1}^N c_mt^{-r-m}.
		\]
		Since the coefficient of $t^{-O_r}$ in $F(x_0)$ is $-P_r(C)$, it vanishes by the assumption $P_r(C)=0$.
		
		Assume that, for some $N\geq1$, the coefficients $c_1,\ldots,c_{N-1}$ have been chosen so that the coefficients of $t^{-O_r},\ldots,t^{-O_r-(N-1)}$ in $F(x_{N-1})$ vanish. Set $\delta_N:=c_Nt^{-r-N}$, so that $x_N=x_{N-1}+\delta_N$. Then
		\begin{equation}\label{eq:FN-recursion-compacta}
			F(x_N)=F(x_{N-1})+\delta_N'-\sum_{i=1}^sA_i(t)\bigl((x_{N-1}+\delta_N)^{n_i}-x_{N-1}^{n_i}\bigr).
		\end{equation}
		
		For each $i\in\{1,\ldots,s\}$, we have $x_{N-1}^{n_i-1}=C^{n_i-1}t^{-(n_i-1)r}+\text{terms of lower order}$. Using $(U+V)^{n_i}-U^{n_i}=n_iU^{n_i-1}V+V^2G_i(U,V)$ for a suitable polynomial $G_i$, we obtain
		\[
		(x_{N-1}+\delta_N)^{n_i}-x_{N-1}^{n_i}=n_iC^{n_i-1}c_Nt^{-n_ir-N}+\text{terms of lower order}.
		\]
		Since $N\geq1$, every term involving $\delta_N^2$ has order strictly lower than $t^{-n_ir-N}$.
		
		It follows from~\eqref{eq:FN-recursion-compacta} that
		\begin{equation}\label{eq:key-coefficient-compacta}
			\operatorname{coeff}_{t^{-O_r-N}}F(x_N)=H_N-\bigl(P_r'(C)+N\mathbf{1}_{\{\partial\in T_r\}}\bigr)c_N,
		\end{equation}
		where $H_N$ depends only on $c_0,\ldots,c_{N-1}$. Indeed, $\delta_N'=-(r+N)c_Nt^{-r-N-1}$ contributes at order $t^{-O_r-N}$ precisely when $\partial\in T_r$. Moreover, the only contribution involving $c_N$ from the $i$-th term on the right-hand side of equation~\eqref{ecu:main} at this order is $-n_i\alpha_iC^{n_i-1}c_N$, and it occurs precisely when $i\in T_r$. Terms indexed by $\ell\notin T_r$ may contribute at later recursive orders. However, since $\Phi_\ell(r)-O_r>0$, their contributions to the coefficient of $t^{-O_r-N}$ involve only coefficients $c_j$ with $j<N$ and are therefore absorbed into $H_N$. Since
		\[
		P_r'(C)=\sum_{i\in T_r\cap\{1,\ldots,s\}}n_i\alpha_iC^{n_i-1}+\mathbf{1}_{\{\partial\in T_r\}}r,
		\]
		formula~\eqref{eq:key-coefficient-compacta} follows.
		
		By hypothesis, $P_r'(C)+N\mathbf{1}_{\{\partial\in T_r\}}\neq0$ for every $N\geq1$. Therefore, vanishing of the coefficient of $t^{-O_r-N}$ uniquely determines $c_N$. Proceeding recursively, we obtain a unique sequence $(c_N)_{N\geq0}$, with $c_0=C$, such that this coefficient vanishes in $F(x_N)$ for every $N\geq0$.
		
		Now set $x(t):=\sum_{m\geq0}c_mt^{-r-m}$. For each fixed $N\geq0$, the coefficient of $t^{-O_r-N}$ in $F(x)$ depends only on $c_0,\ldots,c_N$. Indeed, if $m>N$, the derivative of $c_mt^{-r-m}$ has order strictly lower than $t^{-O_r-N}$ because $r+1=\Phi_\partial(r)\geq O_r$. Similarly, every occurrence of $c_m$ in $A_i(t)x^{n_i}$ has order at most $t^{-(\Phi_i(r)+m)}$, which is strictly lower than $t^{-O_r-N}$ because $\Phi_i(r)\geq O_r$. Hence the coefficient of $t^{-O_r-N}$ in $F(x)$ coincides with that of $F(x_N)$ and therefore vanishes. Thus $F(x)=0$, and uniqueness follows from the recursive construction.
	\end{proof}
	
	If $\partial\in T_r$ and $P_r'(C)+N=0$ for some $N\geq1$, the recursion at level $N$ reduces to the compatibility condition $H_N=0$, and the coefficient $c_N$ is not uniquely determined when this condition holds. If $\partial\notin T_r$ and $P_r'(C)=0$, the linearized recursion is degenerate and the Newton process must be continued after the substitution $x=Ct^{-r}+y$. Thus, in both cases, the leading term need not determine a unique formal solution. A continuation of the Newton process can be used to investigate these exceptional cases, but the resulting higher-order compatibility conditions and formal branches require a separate case-by-case analysis, which we do not pursue here.
	
	For the counting results, we introduce uniform simplicity and nonresonance conditions on the nonzero roots of each edge polynomial.
	
	\begin{definition}\label{def:SN}
		Let $r$ be an edge-admissible value. We say that the edge polynomial $P_r$ satisfies:
		\begin{itemize}
			\item[\textup{(S)}] the simplicity condition if, whenever $\partial\notin T_r$, every nonzero root $C\in\Bbbk$ of $P_r$ satisfies $P_r'(C)\neq0$;
			\item[\textup{(NR)}] the nonresonance condition if, whenever $\partial\in T_r$, every nonzero root $C\in\Bbbk$ of $P_r$ satisfies $P_r'(C)\notin\mathbb{Z}_{<0}$.
		\end{itemize}
	\end{definition}
	
	Condition {\rm (S)} imposes no restriction when $\partial\in T_r$. In particular, nonzero multiple roots are allowed on such edges: if $P_r'(C)=0$, the recursion coefficient in~\eqref{eq:key-coefficient-compacta} is $N\neq0$ for every $N\geq1$. Conversely, condition {\rm (NR)} imposes no restriction when $\partial\notin T_r$.
	
	\begin{corollary}\label{cor:main_bound}
		Let $r$ be edge-admissible, and assume that $P_r$ satisfies conditions {\rm (S)} and {\rm (NR)}. Then equation~\eqref{ecu:main} has at most $\deg P_r-\operatorname{mult}_0(P_r)$ rational solutions with denominator degree $r$, where $\operatorname{mult}_0(P_r)$ denotes the multiplicity of $0$ as a root of $P_r$.
	\end{corollary}
	
	\begin{proof}
		Let $x(t)=1/p(t)$ be a rational solution with $\deg p=r$ and $\operatorname{lc}(p)=c$. By Proposition~\ref{prop:leading_coeff_root}, $1/c$ is a nonzero root of $P_r$.
		
		Suppose that $x(t)=1/p(t)$ and $y(t)=1/q(t)$ are two rational solutions with denominator degree $r$ such that $1/\operatorname{lc}(p)=1/\operatorname{lc}(q)=:C\neq0$. Then both admit Laurent expansions at infinity with leading term $Ct^{-r}$. Since $P_r$ satisfies conditions {\rm (S)} and {\rm (NR)}, Proposition~\ref{prop:existence} gives a unique formal Laurent series solution with this leading term. Hence the Laurent expansions of $x$ and $y$ coincide, and therefore $x(t)\equiv y(t)$.
		
		Consequently, the map $x(t)=1/p(t)\mapsto1/\operatorname{lc}(p)$ is injective on the set of rational solutions with denominator degree $r$, and its image is contained in the set of nonzero roots of $P_r$ in $\Bbbk$. Since $P_r$ has at most $\deg P_r-\operatorname{mult}_0(P_r)$ distinct nonzero roots, the result follows.
	\end{proof}
	
	Equality need not hold, since the formal Laurent series associated with a nonzero root of $P_r$ need not be rational, whereas the solutions counted here are restricted to those of the form $x(t)=1/p(t)$.
	
	\begin{corollary}\label{cor:descartes}
		Assume that $\Bbbk=\mathbb{R}$, let $r$ be edge-admissible, and suppose that $P_r$ satisfies {\rm (S)} and {\rm (NR)}. Then the number of rational solutions of equation~\eqref{ecu:main} with denominator degree $r$ is at most $2(|T_r|-1)$.
	\end{corollary}
	
	\begin{proof} 
		By Corollary~\ref{cor:main_bound}, it suffices to bound the number of distinct nonzero real roots of $P_r$. The polynomial $P_r$ has exactly $|T_r|$ nonzero monomials. Indeed, if $n_1=1$, then $\Phi_1(r)<\Phi_\partial(r)$, so $1$ and $\partial$ cannot both belong to $T_r$. Descartes' rule of signs applied to $P_r(C)$ and $P_r(-C)$ gives at most $|T_r|-1$ positive roots and at most $|T_r|-1$ negative roots, respectively. Hence $P_r$ has at most $2(|T_r|-1)$ distinct nonzero real roots. 
	\end{proof}
	
	\begin{lemma}\label{lem:edge-cardinality-sum}
		Let $\mathcal R$ be a set of distinct edge-admissible values. Then $\mathcal R$ is finite and
		\[
		\sum_{r\in\mathcal R}\bigl(|T_r|-1\bigr)\leq s.
		\]
	\end{lemma}
	
	\begin{proof}
		Each $r\in\mathcal R$ determines a distinct exposed edge $E_r$ of the Newton diagram. Since $\mathcal N$ is the convex hull of the finitely many points in $\mathcal Q$, it has only finitely many edges. Hence $\mathcal R$ is finite.
		
		For each $r\in\mathcal R$, order the points of $\mathcal Q$ lying on $E_r$ according to their position along the edge. These points determine exactly $|T_r|-1$ consecutive segments contained in $E_r$.
		
		The edges $E_r$, with $r\in\mathcal R$, have pairwise disjoint relative interiors and belong to the boundary chain of $\mathcal N$ exposed by the functionals $(a,n)\mapsto-a+nr$, with $r>0$. Consequently, the segments obtained from distinct edges have pairwise disjoint relative interiors, and their union forms an acyclic graph whose vertices belong to $\mathcal Q$.
		
		Since $\mathcal Q=\{Q_1,\ldots,Q_s,Q_\partial\}$ consists of $s+1$ points, this graph has at most $s$ edges. Hence the desired inequality follows.
	\end{proof}
	
	We now combine the edgewise bounds with the geometry of the Newton diagram to obtain global bounds over both fields.
	
	\begin{theorem}\label{thm:global-bound}
		Assume that $P_r$ satisfies conditions {\rm (S)} and {\rm (NR)} for every edge-admissible value $r$. Then:
		\begin{enumerate}[(a)]
			\item Equation~\eqref{ecu:main} has at most $n_s-1$ nonconstant rational solutions.
			\item If $\Bbbk=\mathbb{R}$, then equation~\eqref{ecu:main} has at most $\min\{n_s-1,2s\}$ nonconstant rational solutions.
		\end{enumerate}
	\end{theorem}
	
	\begin{proof}
		For each edge-admissible value $r$, let $\mathcal S_r$ denote the number of rational solutions with denominator degree $r$. By Corollary~\ref{cor:main_bound}, 
		\[ 
		\mathcal S_r\leq\deg P_r-\operatorname{mult}_0(P_r)=\max_{\ell\in T_r}n_\ell-\min_{\ell\in T_r}n_\ell, 
		\] 
		where $n_\partial=1$. The last difference is the vertical length of the exposed edge $E_r$.
		
		Let $r_1<\cdots<r_N$ be the distinct edge-admissible values. If $r<u$, $i\in T_r$, and $j\in T_u$, minimality gives $\Phi_i(r)\leq\Phi_j(r)$ and $\Phi_j(u)\leq\Phi_i(u)$. Adding these inequalities yields $(n_i-n_j)(u-r)\geq0$, and hence $n_i\geq n_j$. Therefore, the vertical projections of the edges $E_{r_1},\ldots,E_{r_N}$ have pairwise disjoint interiors. Hence
		\[
		\sum_{j=1}^N\left(\max_{\ell\in T_{r_j}}n_\ell-\min_{\ell\in T_{r_j}}n_\ell\right)\leq\max_{\ell\in\{1,\ldots,s,\partial\}}n_\ell-\min_{\ell\in\{1,\ldots,s,\partial\}}n_\ell=n_s-1.
		\]
		Consequently,
		\[
		\sum_{j=1}^N\mathcal S_{r_j}\leq n_s-1.
		\]
		
		By Proposition~\ref{prop:leading_coeff_root}, the denominator degree of every rational solution is edge-admissible. Therefore, the total number of rational solutions is at most
		\[
		\sum_r\mathcal S_r\leq\sum_r\bigl(\deg P_r-\operatorname{mult}_0(P_r)\bigr)\leq n_s-1.
		\]
		This proves \textup{(a)}.
		
		Assume now that $\Bbbk=\mathbb{R}$, and let $\mathcal R$ be the set of denominator degrees realized by the rational solutions. By Proposition~\ref{prop:leading_coeff_root}, every $r\in\mathcal R$ is edge-admissible. Corollary~\ref{cor:descartes} gives $\mathcal S_r\leq2(|T_r|-1)$ for every $r\in\mathcal R$. Hence Lemma~\ref{lem:edge-cardinality-sum} yields
		\[
		\sum_{r\in\mathcal R}\mathcal S_r\leq2\sum_{r\in\mathcal R}(|T_r|-1)\leq2s.
		\]
		Combining this estimate with \textup{(a)}, we obtain
		\[
		\sum_{r\in\mathcal R}\mathcal S_r\leq\min\{n_s-1,2s\},
		\]
		which proves \textup{(b)}.
	\end{proof}
	
	\begin{corollary}\label{cor:linear-term-bound}
		Assume that $\Bbbk=\mathbb{R}$ and $n_1=1$, and that $P_r$ satisfies condition {\rm (S)} for every edge-admissible value $r$. Then equation~\eqref{ecu:main} has at most $\min\{n_s-1,2(s-1)\}$ nonconstant rational solutions.
	\end{corollary}
	
	\begin{proof}
		Since $n_1=1$, one has $\Phi_1(r)=r-a_1<r+1=\Phi_\partial(r)$ for every $r>0$. Hence $\partial\notin T_r$ for every edge-admissible value $r$, so condition {\rm (NR)} is vacuous. Moreover, the exposed edges associated with the denominator degrees realized by the rational solutions contain only points from $\{Q_1,\ldots,Q_s\}$. Applying the argument of Lemma~\ref{lem:edge-cardinality-sum} to these $s$ points gives
		\[
		\sum_{r\in\mathcal R}(|T_r|-1)\leq s-1,
		\]
		where $\mathcal R$ denotes the set of realized denominator degrees. Therefore, Corollary~\ref{cor:descartes} yields
		\[
		\sum_{r\in\mathcal R}\mathcal S_r\leq2\sum_{r\in\mathcal R}(|T_r|-1)\leq2(s-1).
		\]
		Theorem~\ref{thm:global-bound}\textup{(a)} also gives at most $n_s-1$ nonconstant rational solutions, and the result follows.
	\end{proof}
	
	The proof of Corollary~\ref{cor:finiteness} shows that each proportionality class contains at most $n_s-1$ nonconstant rational solutions. The next result shows that attaining this bound in a proportionality class forces a composition structure.
	
	\begin{proposition}\label{prop:composition}
		Assume that $1<n_1<\cdots<n_s$, and set $d:=\gcd\{n_1-1,\ldots,n_s-1\}$. If equation~\eqref{ecu:main} has $n_s-1$ distinct nonconstant rational solutions belonging to the same proportionality class, then there exist a nonconstant polynomial $p\in\Bbbk[t]$ and constants $c_i\in\Bbbk^\times$, $i=1,\ldots,s$, such that
		\[
		A_i(t)=c_i p(t)^{n_i-2}p'(t),\qquad i=1,\ldots,s.
		\]
		Consequently, equation~\eqref{ecu:main} satisfies the composition condition with $W(t)=p(t)^d$. In particular, over $\mathbb{R}$, $x=0$ is a composition center on every interval $[a,b]$ such that $p(a)^d=p(b)^d$.
	\end{proposition}
	
	\begin{proof}
		Let $x=1/p$ be one of the rational solutions. Since the $n_s-1$ solutions belong to the same proportionality class, they can be written as $x_j(t)=\lambda_j/p(t)$, $j=1,\ldots,n_s-1$, for distinct $\lambda_j\in\Bbbk^\times$.
		
		By~\eqref{eq:scaled-condition}, the $\lambda_j$ are roots of $
		H(\lambda):=\sum_{i=1}^s A_i p^{n_s-n_i}\bigl(\lambda^{n_i-1}-1\bigr)$. Since $H$ has degree $n_s-1$ and leading coefficient $A_s$, we have $H(\lambda)=A_s\prod_{j=1}^{n_s-1}(\lambda-\lambda_j)$. Write
		\[
		\prod_{j=1}^{n_s-1}(\lambda-\lambda_j)=\sum_{k=0}^{n_s-1}q_k\lambda^k,\qquad q_{n_s-1}=1.
		\]
		Since every $\lambda_j$ is nonzero, $q_0\neq0$.
		
		On the other hand, since $x=1/p$ is a solution, equation~\eqref{eq:denominator} gives $p^{n_s-2}p'+\sum_{i=1}^s A_i p^{n_s-n_i}=0$.
		Hence $H(0)=-\sum_{i=1}^s A_i p^{n_s-n_i}=p^{n_s-2}p'$. Comparing this with $H(0)=q_0A_s$, we obtain $A_s=q_0^{-1}p^{n_s-2}p'$. Comparison of the coefficient of $\lambda^{n_i-1}$ then yields $A_i p^{n_s-n_i}=q_{n_i-1}A_s$, and therefore
		\[
		A_i=\frac{q_{n_i-1}}{q_0}p^{n_i-2}p',\qquad i=1,\ldots,s.
		\]
		Thus the first assertion follows with $c_i=q_{n_i-1}/q_0$.
		
		Now set $W=p^d$. Since $d\mid n_i-1$ for every $i$, we have $W'=dp^{d-1}p'$ and
		\[
		A_i(t)=\frac{c_i}{d}W(t)^{(n_i-1)/d-1}W'(t).
		\]
		Thus $A_i(t)=\widetilde A_i'(W(t))W'(t)$, where $\widetilde A_i(z)=c_i z^{(n_i-1)/d}/(n_i-1)$, and hence equation~\eqref{ecu:main} satisfies the composition condition with $W=p^d$.
		
		If $\Bbbk=\mathbb{R}$ and $p(a)^d=p(b)^d$, then $W(a)=W(b)$, and hence the return map on $[a,b]$ is the identity in a neighborhood of $x=0$. Therefore $x=0$ is a composition center.
	\end{proof}
	
	We conclude this section with an example showing that both bounds in Theorem~\ref{thm:global-bound} are sharp for every $s\geq3$.
	
	\begin{example}\label{ex}
		Fix $s\geq3$ and set $n_i=2i+1$ for $i=1,\ldots,s$. Let $\varepsilon>0$ be specified below and set $\lambda_j=1+j\varepsilon$ for $j=1,\ldots,s$. Define
		\[
		P(C):=\frac{(-1)^s}{\lambda_1^2\cdots\lambda_s^2}C\prod_{j=1}^s(C^2-\lambda_j^2).
		\]
		Writing $e_k=e_k(\lambda_1^2,\ldots,\lambda_s^2)$ for the
		$k$-th elementary symmetric polynomial, with $e_0=1$, we have
		\[
		P(C)=C+\sum_{i=1}^s(-1)^i\frac{e_{s-i}}{e_s}C^{2i+1}.
		\]
		Thus, consider equation~\eqref{ecu:main} with $A_i(t)=(-1)^i(e_{s-i}/e_s)t^{2i-1}$ for $i=1,\ldots,s$. Since $e_k>0$, each $A_i$ is nonzero.
		
		Then $a_i=2i-1=n_i-2$, so the exponent points $Q_i=(2i-1,2i+1)$, $i=1,\ldots,s$, together with $Q_\partial=(-1,1)$, lie on the same exposed edge associated with $r=1$. Thus $r=1$ is the only edge-admissible value and $T_1=\{1,\ldots,s,\partial\}$. The corresponding edge polynomial is
		\[
		P_1(C)=\frac{(-1)^s}{\lambda_1^2\cdots\lambda_s^2}C\prod_{j=1}^s(C^2-\lambda_j^2),
		\]
		whose nonzero roots are $\pm\lambda_1,\ldots,\pm\lambda_s$.
		
		For each $j=1,\ldots,s$, direct substitution shows that $x_{j,\pm}(t)=\pm\lambda_j/t$ is a rational solution of equation~\eqref{ecu:main}. Equivalently, $x_{j,\pm}=1/p_{j,\pm}$ with $p_{j,\pm}(t)=\pm t/\lambda_j$. Hence the equation has $2s$ distinct nonconstant real rational solutions, all with denominator degree one.
		
		At a nonzero root $\pm\lambda_k$,
		\[
		P_1'(\pm\lambda_k)=\frac{2(-1)^s\lambda_k^2}{\lambda_1^2\cdots\lambda_s^2}\prod_{j\neq k}(\lambda_k^2-\lambda_j^2),
		\]
		which is nonzero because the $\lambda_j$ are distinct. Since $P_1'(\pm\lambda_k)\to0$ as $\varepsilon\to0^+$ for every $k$, we may choose $\varepsilon>0$ sufficiently small so that $0<|P_1'(\pm\lambda_k)|<1$ for all $k$. Thus conditions {\rm (S)} and {\rm (NR)} are satisfied.
		
		Finally, $n_s=2s+1$, and hence $n_s-1=2s$. Therefore the $2s$ solutions above attain simultaneously the complex bound $n_s-1$ and the real bound $\min\{n_s-1,2s\}$ in Theorem~\ref{thm:global-bound}.
	\end{example}
	
	The $2s=n_s-1$ rational solutions in Example~\ref{ex} belong to the same proportionality class. Hence Proposition~\ref{prop:composition} applies. In this case $d=2$ and $p(t)=t$, so $W(t)=t^2$, and the equation defines a composition center on every symmetric interval $[-T,T]$.
	
	\section{Denominator-degree multiplicity profiles}
	
	Throughout this section, we specialize equation~\eqref{ecu:main} to the three-term case
	\[
	x'=A_3(t)x^{n_3}+A_2(t)x^{n_2}+A_1(t)x^{n_1},\qquad 1<n_1<n_2<n_3.
	\]
	Although the Newton--Puiseux method and the global bounds established in Sections~\ref{Sect:2} and~\ref{Sect:3} apply to an arbitrary finite number of terms, the small number of exponent points in the three-term case allows us to classify the configurations of the realized denominator degrees and obtain upper bounds for the corresponding multiplicities.
	
	\subsection{Realized denominator degrees}
	
	We first prove that at most three distinct denominator degrees can occur.
	
	\begin{proposition}\label{prop:three-distinct-degrees}
		Suppose that equation~\eqref{ecu:main} has three nonconstant rational solutions $x_j=1/p_j(t)$ whose denominator degrees satisfy $d_1:=\deg p_1<d_2:=\deg p_2<d_3:=\deg p_3$. Then
		\begin{align}
			a_1&=(n_1-1)d_3-1,\label{eq:three-degrees-a1}\\
			a_2&=(n_1-1)d_3-1+(n_2-n_1)d_2,\label{eq:three-degrees-a2}\\
			a_3&=(n_1-1)d_3-1+(n_2-n_1)d_2+(n_3-n_2)d_1.\label{eq:three-degrees-a3}
		\end{align}
		Consequently,
		\[
		d_1=\frac{a_3-a_2}{n_3-n_2},\qquad d_2=\frac{a_2-a_1}{n_2-n_1},\qquad d_3=\frac{a_1+1}{n_1-1},
		\]
		and the corresponding exposed edges satisfy $T_{d_1}=\{3,2\}$, $T_{d_2}=\{2,1\}$, and $T_{d_3}=\{1,\partial\}$.
	\end{proposition}
	
	\begin{proof}
		Set $a:=n_3-n_2$, $b:=n_3-n_1$, and $c:=n_3-1$, and consider the matrix
		\[
		G:=
		\begin{pmatrix}
			1 & p_1^a & p_1^b\\
			1 & p_2^a & p_2^b\\
			1 & p_3^a & p_3^b
		\end{pmatrix}.
		\]
		Its determinant is
		\[
		\Delta=(p_2^a-p_1^a)(p_3^b-p_1^b)-(p_3^a-p_1^a)(p_2^b-p_1^b).
		\]
		The six terms in the expansion of $\Delta$ have degrees $ad_2+bd_3$, $ad_1+bd_2$, $ad_3+bd_1$, $ad_3+bd_2$, $ad_1+bd_3$, and $ad_2+bd_1$. Since $d_1<d_2<d_3$ and $0<a<b$, the unique term of maximal degree has degree $ad_2+bd_3$. Hence $\Delta\not\equiv0$ and $\deg\Delta=ad_2+bd_3$.
		
		For each $j\in\{1,2,3\}$, equation~\eqref{eq:denominator} gives $A_3+A_2p_j^a+A_1p_j^b=-p_j^{n_3-2}p_j'$. Thus $(A_3,A_2,A_1)$ is the unique solution of a linear system with coefficient matrix $G$. By Cramer's rule, the numerators defining $A_1$, $A_2$, and $A_3$ have respective degrees $ad_2+cd_3-1$, $bd_2+cd_3-1$, and $ad_1+bd_2+cd_3-1$. In each case, the term of maximal degree is unique, so no cancellation occurs. Subtracting $\deg\Delta=ad_2+bd_3$ gives \eqref{eq:three-degrees-a1}, \eqref{eq:three-degrees-a2}, and~\eqref{eq:three-degrees-a3}.
		
		The identities~\eqref{eq:three-degrees-a1}, \eqref{eq:three-degrees-a2}, and~\eqref{eq:three-degrees-a3} give
		\[
		d_3=\frac{a_1+1}{n_1-1},\qquad d_2=\frac{a_2-a_1}{n_2-n_1},\qquad d_1=\frac{a_3-a_2}{n_3-n_2}.
		\]
		Moreover, the preceding identities give, for every $r>0$,
		\[
		\begin{aligned}
			\Phi_3(r)-\Phi_2(r)&=(n_3-n_2)(r-d_1),\\
			\Phi_2(r)-\Phi_1(r)&=(n_2-n_1)(r-d_2),\\
			\Phi_1(r)-\Phi_\partial(r)&=(n_1-1)(r-d_3).
		\end{aligned}
		\]
		Since $d_1<d_2<d_3$ and $1<n_1<n_2<n_3$, evaluating these
		differences at $r=d_1,d_2,d_3$ identifies the minimizing indices
		and gives $T_{d_1}=\{3,2\}$, $T_{d_2}=\{2,1\}$, and
		$T_{d_3}=\{1,\partial\}$.
	\end{proof}
	
	\begin{corollary}\label{cor:at-most-three-degrees}
		The denominators of the nonconstant rational solutions of equation~\eqref{ecu:main} realize at most three distinct degrees.
	\end{corollary}
	
	\begin{proof}
		Suppose that four distinct denominator degrees $d_1<d_2<d_3<d_4$ were realized. Applying Proposition~\ref{prop:three-distinct-degrees} to the degrees $d_1<d_2<d_3$ gives $a_1=(n_1-1)d_3-1$, whereas applying it to the degrees $d_2<d_3<d_4$ gives $a_1=(n_1-1)d_4-1$. Hence $d_3=d_4$, a contradiction.
	\end{proof}
	
	\begin{corollary}\label{cor:at-most-two-degrees}
		Assume that $P_r$ satisfies condition {\rm (NR)} for every edge-admissible value $r$. Then the nonconstant rational solutions of equation~\eqref{ecu:main} realize at most two distinct denominator degrees.
	\end{corollary}
	
	\begin{proof}
		Suppose that three distinct degrees $d_1<d_2<d_3$ were realized. By Proposition~\ref{prop:three-distinct-degrees}, one has $T_{d_3}=\{1,\partial\}$, so $P_{d_3}(C)=\alpha_1C^{n_1}+d_3C$.
		
		Let $x=1/p$ be a rational solution with $\deg p=d_3$, and set $C:=1/\operatorname{lc}(p)$. By Proposition~\ref{prop:leading_coeff_root}, $C$ is a nonzero root of $P_{d_3}$. Hence $\alpha_1C^{n_1-1}=-d_3$, and therefore
		\[
		P_{d_3}'(C)=n_1\alpha_1C^{n_1-1}+d_3=-(n_1-1)d_3\in\mathbb Z_{<0}.
		\]
		This contradicts condition {\rm (NR)}. Hence at most two distinct denominator degrees can be realized.
	\end{proof}
	
	We next establish the algebraic identity needed to determine the possible configurations when exactly two denominator degrees are realized.
	
	\begin{proposition}\label{prop:two-solutions-identity}
		Let $x_j=1/p_j(t)$, $j=1,2$, be two nonconstant rational solutions such that $p_1^{n_3-n_2}-p_2^{n_3-n_2}\not\equiv0$. Then
		\[
		A_2=-\frac{\frac{1}{n_3-1}\bigl(p_1^{n_3-1}-p_2^{n_3-1}\bigr)'+A_1\bigl(p_1^{n_3-n_1}-p_2^{n_3-n_1}\bigr)}{p_1^{n_3-n_2}-p_2^{n_3-n_2}}.
		\]
	\end{proposition}
	
	\begin{proof}
		For $j=1,2$, equation~\eqref{eq:denominator} gives $A_3+A_2p_j^{n_3-n_2}+A_1p_j^{n_3-n_1}=-p_j^{n_3-2}p_j'$. Subtracting the two identities, using $p_j^{n_3-2}p_j'=(p_j^{n_3-1})'/(n_3-1)$, and dividing by the nonzero polynomial $p_1^{n_3-n_2}-p_2^{n_3-n_2}$ gives the result.
	\end{proof}
	
	When the two denominators have distinct degrees, each difference of
	powers appearing in Proposition~\ref{prop:two-solutions-identity}
	has the degree determined by the denominator of larger degree.
	
	\begin{lemma}\label{lem:maximal-degree-trichotomy}
		Let $x_j=1/p_j(t)$, $j=1,2$, be two nonconstant rational solutions such that $\deg p_2<\deg p_1$, and set $d:=\deg p_1$. Then exactly one of the following alternatives holds:
		\begin{enumerate}[(C1)]
			\item $a_1<(n_1-1)d-1$ and $a_2=(n_2-1)d-1$.
			\item $a_1>(n_1-1)d-1$ and $a_2=a_1+(n_2-n_1)d$.
			\item $a_1=(n_1-1)d-1$ and $a_2\leq(n_2-1)d-1$.
		\end{enumerate}
	\end{lemma}
	
	\begin{proof}
		Since $\deg p_2<d$, one has $\deg(p_1^m-p_2^m)=md$ for every positive integer $m$. In particular, the denominator in Proposition~\ref{prop:two-solutions-identity} has degree $(n_3-n_2)d$, the derivative term in its numerator has degree $(n_3-1)d-1$, and the term involving $A_1$ has degree $a_1+(n_3-n_1)d$.
		
		If $a_1<(n_1-1)d-1$, then the derivative term has strictly larger degree, and hence $a_2=(n_2-1)d-1$. If $a_1>(n_1-1)d-1$, then the term involving $A_1$ has strictly larger degree, and hence $a_2=a_1+(n_2-n_1)d$. Finally, if $a_1=(n_1-1)d-1$, then the two terms in the numerator have the same degree and their leading terms may cancel. Since $A_2\neq0$, the numerator does not vanish identically, and therefore $a_2\leq(n_2-1)d-1$. The three alternatives are mutually exclusive and exhaustive.
	\end{proof}
	
	We now classify the configurations that can occur when exactly two denominator degrees are realized. The only additional assumption required for this classification is condition {\rm (NR)} on the edges containing the derivative point. Indeed, if $T_r=\{i,\partial\}$ for some $i\in\{1,2,3\}$, then every nonzero root $C$ of $P_r$ satisfies $P_r'(C)=-(n_i-1)r\in\mathbb{Z}_{<0}$. By Proposition~\ref{prop:leading_coeff_root}, condition {\rm (NR)} therefore prevents any rational solution from being associated with such a binomial edge.
	
	\begin{proposition}\label{prop:two-degree-configurations}
		Assume that $P_r$ satisfies condition {\rm (NR)} for every edge-admissible value $r$. Suppose that the rational solutions of equation~\eqref{ecu:main} realize exactly two distinct denominator degrees $d_1<d_2$. Then $T_{d_1}=\{3,2\}$, and exactly one of the following alternatives holds:
		\begin{enumerate}[(i)]
			\item $T_{d_2}=\{2,1\}$, $d_1=(a_3-a_2)/(n_3-n_2)$, and $d_2=(a_2-a_1)/(n_2-n_1)$.
			\item $T_{d_2}=\{2,1,\partial\}$, $d_1=(a_3-a_2)/(n_3-n_2)$, and
			\[
			d_2=\frac{a_2-a_1}{n_2-n_1}=\frac{a_1+1}{n_1-1}.
			\]
		\end{enumerate}
	\end{proposition}
	
	\begin{proof}
		Choose rational solutions $x_j=1/p_j$, $j=1,2$, such that $\deg p_2=d_2$ and $\deg p_1=d_1<d_2$. Lemma~\ref{lem:maximal-degree-trichotomy} applies with $d=d_2$.
		
		Suppose first that {\rm (C1)} holds. Then $a_2=(n_2-1)d_2-1$ and $a_1<(n_1-1)d_2-1$. Applying equation~\eqref{eq:denominator} to $p_2$, the terms $A_2p_2^{n_3-n_2}$ and $p_2^{n_3-2}p_2'$ have degree $(n_3-1)d_2-1$, whereas $A_1p_2^{n_3-n_1}$ has smaller degree. Hence $a_3\leq(n_3-1)d_2-1$. If the inequality were strict, then $T_{d_2}=\{2,\partial\}$, which is excluded by the nonresonance assumption. Thus $a_3=(n_3-1)d_2-1$, and hence $T_{d_2}=\{3,2,\partial\}$. Since this edge contains both $Q_3$ and $Q_\partial$, its vertical length is $n_3-1$, so no other edge-admissible value can occur. This contradicts the realization of $d_1$.
		
		Suppose next that {\rm (C2)} holds. Then $a_2=a_1+(n_2-n_1)d_2$, so $d_2=(a_2-a_1)/(n_2-n_1)$ and $\Phi_2(d_2)=\Phi_1(d_2)$. Moreover, $a_1>(n_1-1)d_2-1$ is equivalent to $\Phi_1(d_2)<\Phi_\partial(d_2)$. Applying equation~\eqref{eq:denominator} to $p_2$ gives $a_3\leq a_1+(n_3-n_1)d_2$.
		
		If equality holds, then $T_{d_2}=\{3,2,1\}$. The only possible additional edge of the Newton diagram at infinity joins $Q_1$ and $Q_\partial$, and its associated value is $(a_1+1)/(n_1-1)>d_2$. Therefore no edge-admissible value smaller than $d_2$ can occur, contradicting the realization of $d_1$.
		
		Consequently, $a_3<a_1+(n_3-n_1)d_2$, and therefore $T_{d_2}=\{2,1\}$. Since $d_1<d_2$ is another realized edge-admissible value, the preceding edge in the relevant boundary chain must join $Q_3$ and $Q_2$. Hence $T_{d_1}=\{3,2\}$ and $d_1=(a_3-a_2)/(n_3-n_2)$. This gives {\rm (i)}.
		
		Finally, suppose that {\rm (C3)} holds. Then $a_1=(n_1-1)d_2-1$ and $a_2\leq(n_2-1)d_2-1$, so $\Phi_1(d_2)=\Phi_\partial(d_2)$ and $\Phi_2(d_2)\geq\Phi_1(d_2)$. Applying equation~\eqref{eq:denominator} to $p_2$ gives $a_3\leq(n_3-1)d_2-1$, and hence $\Phi_3(d_2)\geq\Phi_1(d_2)$.
		
		If both inequalities are strict, then $T_{d_2}=\{1,\partial\}$, which is excluded by the nonresonance assumption. If $a_2<(n_2-1)d_2-1$ and $a_3=(n_3-1)d_2-1$, then $T_{d_2}=\{3,1,\partial\}$. If both equalities hold, then $T_{d_2}=\{3,2,1,\partial\}$. In either case, the exposed edge contains both $Q_3$ and $Q_\partial$, so its vertical length is $n_3-1$ and no other edge-admissible value can occur.
		
		Thus one must have $a_2=(n_2-1)d_2-1$ and $a_3<(n_3-1)d_2-1$. It follows that $T_{d_2}=\{2,1,\partial\}$ and
		\[
		d_2=\frac{a_2-a_1}{n_2-n_1}=\frac{a_1+1}{n_1-1}.
		\]
		Since $d_1<d_2$ is another realized edge-admissible value, the preceding edge must join $Q_3$ and $Q_2$. Hence $T_{d_1}=\{3,2\}$ and $d_1=(a_3-a_2)/(n_3-n_2)$. This gives {\rm (ii)}.
	\end{proof}
	
	The classification also yields a refinement over either base field when the edge associated with the maximal realized denominator degree contains the three nonlinear exponent points but not the derivative point.
	
	\begin{corollary}\label{cor:three-nonlinear-edge}
		Assume that $P_r$ satisfies conditions {\rm (S)} and {\rm (NR)} for every edge-admissible value $r$, and suppose that equation~\eqref{ecu:main} has at least one nonconstant rational solution. Let $d$ be the maximal realized denominator degree. If $T_d=\{3,2,1\}$, then every nonconstant rational solution has denominator degree $d$, and equation~\eqref{ecu:main} has at most $n_3-n_1$ such solutions.
	\end{corollary}
	
	\begin{proof}
		Since $T_d=\{3,2,1\}$, the points $Q_3,Q_2,Q_1$ lie on the same exposed edge. The only possible additional edge in the relevant boundary chain joins $Q_1$ and $Q_\partial$. Its associated value is $(a_1+1)/(n_1-1)$. Since $\partial\notin T_d$, one has $\Phi_1(d)<\Phi_\partial(d)$, which is equivalent to
		\[
		\frac{a_1+1}{n_1-1}>d.
		\]
		Thus any other edge-admissible value is larger than $d$. Since $d$ is the maximal realized denominator degree, every nonconstant rational solution has denominator degree $d$.
		
		The corresponding edge polynomial is
		\[
		P_d(C)=\alpha_3C^{n_3}+\alpha_2C^{n_2}+\alpha_1C^{n_1}.
		\]
		Hence $\deg P_d=n_3$ and $\operatorname{mult}_0(P_d)=n_1$. Corollary~\ref{cor:main_bound} therefore gives at most $n_3-n_1$ nonconstant rational solutions.
	\end{proof}
	
	\subsection{Real multiplicity profiles}
	
	We now establish componentwise upper bounds for the real denominator-degree multiplicity profiles under the simplicity {\rm (S)} and nonresonance {\rm (NR)} conditions of Section~\ref{Sect:3}.
	
	Let $\mathcal D$ denote the set of denominator degrees realized by the nonconstant rational solutions of equation~\eqref{ecu:main}, that is,
	\[
	\mathcal D:=\{\deg p:x=1/p\text{ is a nonconstant rational solution of equation~\eqref{ecu:main}}\}.
	\]
	For each $d\in\mathcal D$, let $\nu_d$ denote the number of rational solutions with denominator degree $d$. If $\mathcal D=\{d_1<\cdots<d_m\}$, we call $(\nu_{d_1},\ldots,\nu_{d_m})$ the denominator-degree multiplicity profile of equation~\eqref{ecu:main}.
	
	\begin{theorem}\label{thm:real-basic-profiles}
		Assume that $\Bbbk=\mathbb{R}$, that $P_r$ satisfies conditions {\rm (S)} and {\rm (NR)} for every edge-admissible value $r$, and that equation~\eqref{ecu:main} has at least one nonconstant rational solution. Then the denominator-degree multiplicity profile of equation~\eqref{ecu:main} is componentwise bounded by one of $(6)$, $(2,2)$, or $(2,4)$. More precisely:
		\begin{enumerate}[(i)]
			\item If exactly one denominator degree $d$ is realized, then $\nu_d\leq2(|T_d|-1)\leq6$.
			\item If two denominator degrees $d_1<d_2$ are realized and $T_{d_2}=\{2,1\}$, then $\nu_{d_1}\leq2$ and $\nu_{d_2}\leq2$.
			\item If two denominator degrees $d_1<d_2$ are realized and $T_{d_2}=\{2,1,\partial\}$, then $\nu_{d_1}\leq2$ and $\nu_{d_2}\leq4$.
		\end{enumerate}
	\end{theorem}
	
	\begin{proof}
		By Corollary~\ref{cor:at-most-two-degrees}, at most two distinct denominator degrees can be realized.
		
		Suppose first that exactly one denominator degree $d$ is realized. By Corollary~\ref{cor:descartes}, one has $\nu_d\leq2(|T_d|-1)$. Since $T_d\subseteq\{1,2,3,\partial\}$, it follows that $|T_d|\leq4$, and hence $\nu_d\leq6$.
		
		Suppose now that two denominator degrees $d_1<d_2$ are realized. By Proposition~\ref{prop:two-degree-configurations}, one has $T_{d_1}=\{3,2\}$, and either $T_{d_2}=\{2,1\}$ or $T_{d_2}=\{2,1,\partial\}$. Since $|T_{d_1}|=2$, Corollary~\ref{cor:descartes} gives $\nu_{d_1}\leq2$.
		
		If $T_{d_2}=\{2,1\}$, then $|T_{d_2}|=2$, and Corollary~\ref{cor:descartes} gives $\nu_{d_2}\leq2$. Thus the multiplicity profile is componentwise bounded by $(2,2)$.
		
		If $T_{d_2}=\{2,1,\partial\}$, then $|T_{d_2}|=3$, and Corollary~\ref{cor:descartes} gives $\nu_{d_2}\leq4$. Thus the multiplicity profile is componentwise bounded by $(2,4)$.
	\end{proof}
	
	Taking $s=3$ in Example~\ref{ex} yields six nonconstant real rational solutions with the same denominator degree. Hence the multiplicity profile $(6)$ in Theorem~\ref{thm:real-basic-profiles} is attained.
	
	The preceding multiplicity bounds can be improved by excluding antipodal pairs of roots in the two edge configurations where they may attain the maximal real contribution.
	
	\begin{definition}\label{def:antipodal-separation} 
		Assume that $\Bbbk=\mathbb{R}$. We say that equation~\eqref{ecu:main} satisfies the antipodal separation condition {\rm (AS)} if, for every edge-admissible value $r$ such that either $T_r=\{3,2\}$ and $n_3-n_2$ is even, or $T_r=\{3,2,1,\partial\}$ and $n_1,n_2,n_3$ are all odd, no two nonzero real roots of $P_r$ are opposite. 
	\end{definition}
	
	In both configurations appearing in condition {\rm (AS)}, the reduced edge polynomial is even. Indeed, in the first case,
	\[
	C^{-n_2}P_r(C)=\alpha_3C^{n_3-n_2}+\alpha_2,
	\]
	whereas in the second case,
	\[
	C^{-1}P_r(C)=\alpha_3C^{n_3-1}
	+\alpha_2C^{n_2-1}+\alpha_1C^{n_1-1}+r.
	\]
	Thus every nonzero real root occurs together with its opposite. Consequently, condition {\rm (AS)} is equivalent to requiring that $P_r$ have no nonzero real roots for each of these configurations. Under condition {\rm (AS)}, Proposition~\ref{prop:leading_coeff_root} implies  that no nonconstant real rational solution can be associated with such an edge. In particular, if a realized denominator degree $d$ satisfies $T_d=\{3,2,1,\partial\}$, then at least one of $n_1,n_2,n_3$ is even.
	
	\begin{corollary}\label{cor:real-multiplicity-profiles}
		Assume that $\Bbbk=\mathbb{R}$, that $P_r$ satisfies conditions {\rm (S)} and {\rm (NR)} for every edge-admissible value $r$, and that equation~\eqref{ecu:main} satisfies {\rm (AS)}. If the equation has at least one nonconstant rational solution, then its denominator-degree multiplicity profile is componentwise bounded by one of $(5)$, $(1,2)$, or $(1,4)$. Moreover, if exactly two denominator degrees $d_1<d_2$ are realized, then $n_3-n_2$ is odd and $\nu_{d_1}=1$.
	\end{corollary}
	
	\begin{proof}
		By Theorem~\ref{thm:real-basic-profiles}, it suffices to improve the bounds corresponding to the profiles $(6)$, $(2,2)$, and $(2,4)$.
		
		Suppose first that exactly one denominator degree $d$ is realized. The only case not already bounded by $4$ in Theorem~\ref{thm:real-basic-profiles} is $T_d=\{3,2,1,\partial\}$. Then $P_d(C)=C\widetilde P_d(C)$, where
		\[
		\widetilde P_d(C)=\alpha_3C^{n_3-1}+\alpha_2C^{n_2-1}+\alpha_1C^{n_1-1}+d.
		\]
		Descartes' rule gives at most three positive roots and at most three negative roots. If $\widetilde P_d$ had six distinct nonzero real roots, then it would have exactly three positive and three negative roots. Consequently, both $\widetilde P_d(C)$ and $\widetilde P_d(-C)$ would have three sign changes. Since their constant terms have the same positive sign, this is possible only if $n_1-1$, $n_2-1$, and $n_3-1$ are all even, equivalently, if $n_1,n_2,n_3$ are all odd. In that case $\widetilde P_d$ is even, so its nonzero real roots occur in opposite pairs, contrary to condition {\rm (AS)}. Hence $\widetilde P_d$ has at most five distinct nonzero real roots, and Corollary~\ref{cor:main_bound} gives $\nu_d\leq5$.
		
		Suppose now that two denominator degrees $d_1<d_2$ are realized. By Proposition~\ref{prop:two-degree-configurations}, one has $T_{d_1}=\{3,2\}$, and the corresponding reduced edge polynomial is
		\[
		\widetilde P_{d_1}(C)
		=\alpha_3C^{n_3-n_2}+\alpha_2.
		\]
		Since $d_1$ is realized, Proposition~\ref{prop:leading_coeff_root} implies that $\widetilde P_{d_1}$ has a nonzero real root. If $n_3-n_2$ were even, the opposite of this root would also be a root, contradicting condition {\rm (AS)}. Hence $n_3-n_2$ is odd, and $\widetilde P_{d_1}$ has exactly one real root, which is nonzero. Corollary~\ref{cor:main_bound} therefore gives $\nu_{d_1}\leq1$, and the realization of $d_1$ yields $\nu_{d_1}=1$.
		
		The bounds $\nu_{d_2}\leq2$ when $T_{d_2}=\{2,1\}$ and $\nu_{d_2}\leq4$ when $T_{d_2}=\{2,1,\partial\}$ follow from Theorem~\ref{thm:real-basic-profiles}. Therefore, the corresponding profiles are componentwise bounded by $(1,2)$ and $(1,4)$, respectively.
	\end{proof}
	
	\begin{corollary}\label{cor:sharp-real-bound}
		Assume that $\Bbbk=\mathbb{R}$, that $P_r$ satisfies conditions {\rm (S)} and {\rm (NR)} for every edge-admissible value $r$, and that equation~\eqref{ecu:main} satisfies {\rm (AS)}. Then equation~\eqref{ecu:main} has at most $\min\{n_3-1,5\}$ nonconstant rational solutions.
	\end{corollary}
	
	\begin{proof}
		If the equation has no nonconstant rational solutions, the conclusion is immediate. Otherwise, Corollary~\ref{cor:real-multiplicity-profiles} gives at most five rational solutions, while Theorem~\ref{thm:global-bound} gives at most $n_3-1$ rational solutions.
	\end{proof}
	
	\begin{corollary}\label{cor:five-solutions}
		Assume that $\Bbbk=\mathbb{R}$, that $P_r$ satisfies conditions {\rm (S)} and {\rm (NR)} for every edge-admissible value $r$, and that equation~\eqref{ecu:main} satisfies {\rm (AS)}. If the equation has five nonconstant rational solutions, then exactly one of the following holds:
		\begin{enumerate}[(i)]
			\item All five solutions have the same denominator degree $d$, and $T_d=\{3,2,1,\partial\}$.
			\item Exactly two denominator degrees $d_1<d_2$ occur, their multiplicity profile is $(1,4)$, and $T_{d_1}=\{3,2\}$ and $T_{d_2}=\{2,1,\partial\}$.
		\end{enumerate}
	\end{corollary}
	
	\begin{proof}
		By Corollary~\ref{cor:real-multiplicity-profiles}, five solutions can occur only in the one-degree case with multiplicity five or in the two-degree case with multiplicity profile $(1,4)$.
		
		In the first case, all five solutions have the same denominator degree $d$. If $|T_d|\leq3$, then Corollary~\ref{cor:descartes} gives $\nu_d\leq2(|T_d|-1)\leq4$, a contradiction. Hence $|T_d|=4$, and therefore $T_d=\{3,2,1,\partial\}$.
		
		In the second case, Proposition~\ref{prop:two-degree-configurations} and Corollary~\ref{cor:real-multiplicity-profiles} give $T_{d_1}=\{3,2\}$ and $T_{d_2}=\{2,1,\partial\}$, with multiplicities exactly one and four.
	\end{proof}
	
	The following example shows that the bound in Theorem~\ref{thm:global-bound}\textup{(a)} is attained in the three-term case over both $\mathbb{R}$ and $\mathbb{C}$. It also shows that the absolute constant $5$ in Corollary~\ref{cor:sharp-real-bound} cannot be reduced.
	
	\begin{example}\label{ex:sharp-real}
		Let $(n_1,n_2,n_3)=(2,4,6)$ and consider equation~\eqref{ecu:main} with $A_1(t)=23t/3$, $A_2(t)=-7t^5$, and $A_3(t)=4t^9/3$. Then $(a_1,a_2,a_3)=(1,5,9)$, and the four exponent points lie on the edge associated with $r=2$, so that $T_2=\{1,2,3,\partial\}$. The corresponding edge polynomial is
		\[
		P_2(C)=\frac{4}{3}C^6-7C^4+\frac{23}{3}C^2+2C=\frac{1}{3}C(C+1)(C+2)(2C-3)(2C^2-3C-1).
		\]
		Its five nonzero roots are $-2$, $-1$, $3/2$, $(3-\sqrt{17})/4$, and $(3+\sqrt{17})/4$, and they give the rational solutions
		\[
		-\frac{2}{t^2},\qquad -\frac{1}{t^2},\qquad \frac{3}{2t^2},\qquad \frac{3-\sqrt{17}}{4t^2},\qquad \frac{3+\sqrt{17}}{4t^2}.
		\]
		
		The values of $P_2'$ at these roots, in the same order, are
		\[
		-\frac{182}{3},\qquad \frac{20}{3},\qquad -\frac{35}{4},\qquad \frac{51}{8}-\frac{47\sqrt{17}}{24},\qquad \frac{51}{8}+\frac{47\sqrt{17}}{24}.
		\]
		They are all nonzero, and none is a negative integer. Thus $P_2$ satisfies conditions {\rm (S)} and {\rm (NR)}. Condition {\rm (AS)} is also satisfied, since the only relevant four-point configuration would require $n_1,n_2,n_3$ to be all odd. Therefore, the equation satisfies {\rm (S)}, {\rm (NR)}, and {\rm (AS)} and has five rational solutions with the same denominator degree. Since $s=3$ and $n_3=6$, both the complex bound $n_3-1$ and the real bounds $\min\{n_3-1,2s\}$ and $\min\{n_3-1,5\}$ equal $5$. Hence the example attains all three bounds and shows, in particular, that the absolute constant $5$ in Corollary~\ref{cor:sharp-real-bound} cannot be reduced.
	\end{example}
	
	In the case $s=3$ of Example~\ref{ex}, the exponents $n_1,n_2,n_3$ are all odd and the nonzero roots of $P_1$ occur in opposite pairs, so condition {\rm (AS)} is not satisfied. Hence conditions {\rm (S)} and {\rm (NR)} alone do not imply an upper bound of five nonconstant real rational solutions. Together with Example~\ref{ex:sharp-real}, this shows that the bounds of six under {\rm (S)} and {\rm (NR)}, and five under the additional condition {\rm (AS)}, are both sharp: the former is attained through an antipodal configuration of roots, whereas the latter remains sharp after such configurations are excluded. Moreover, the five rational solutions in Example~\ref{ex:sharp-real} belong to the same proportionality class and attain the maximal size $n_3-1$. Hence Proposition~\ref{prop:composition} applies, showing that this extremal example also satisfies the composition condition.
	
	\section*{Acknowledgements} 
	
	The authors are grateful to Prof. Manuel Fern\'andez for a careful reading of the introduction and for constructive comments that helped improve the exposition and make the main ideas more accessible.
	

\end{document}